\documentclass[<longtitle>]{elsarticle}
\usepackage{graphicx}
\usepackage{amssymb,amsmath,amsfonts,amsthm,amscd,mathrsfs, epstopdf, amsbsy, psfrag, color,diagrams}
\usepackage{wasysym}
\usepackage[all]{xy}
\numberwithin{equation}{section}

\newtheorem{theorem}[equation]{Theorem}
\newtheorem{lemma}[equation]{Lemma}
\newtheorem{corollary}[equation]{Corollary}
\newtheorem{proposition}[equation]{Proposition}
\theoremstyle{definition}

\newtheorem{iesetup}[equation]{Index and Period Setup}
\theoremstyle{remark}

\newtheorem{remark}[equation]{Remark}

\newcommand{\per}{\mathop{\rm per}}
\newcommand{\Q}{\ensuremath{\mathbb{Q}}}
\newcommand{\Z}{\ensuremath{\mathbb{Z}}}
\renewcommand{\P}{\ensuremath{\mathbb{P}}}

\renewcommand{\H}{\ensuremath{\mathrm{H}}}

\newcommand{\Hom}{\ensuremath{\mathop{\mathrm{Hom}}}}
\newcommand{\chara}{\ensuremath{\mathop{\mathrm{char}}}}
\newcommand{\cor}{\ensuremath{\mathop{\mathrm{cor}}\nolimits}}

\newcommand{\Br}{\ensuremath{\mathop{\mathrm{Br}}}}
\newcommand{\lcm}{\ensuremath{\mathop{\mathrm{lcm}}}}
\newcommand{\res}{\ensuremath{\mathop{\mathrm{res}}}}

\newcommand{\slashfrac}[2]{\ensuremath{\raise1ex\hbox{#1}\kern-.2em/\kern-.30em\lower1ex\hbox{#2}}}
\newcommand{\ind}{\ensuremath{\mathop{\mathrm{ ind }}}}

\newcommand{\inv}{\ensuremath{\mathop{\mathrm{inv}}}}

\def\df={\buildrel \textrm{df}\over =}
\def\Frac{\mathop{\textrm{Frac}}\nolimits}
\def\Gal{\mathop{\textrm{Gal}}\nolimits}
\def\Rev{\mathop{\textrm{Rev}}\nolimits}
\def\id{\mathord{\textrm{id}}}
\def\unit{{\scriptstyle\mathord{\times}}}
\def\Spec{\mathop{\textrm{Spec}}\nolimits}
\def\Proj{\mathop{\textrm{Proj}}\nolimits}

\def\GG{\mathbb{G}}
\def\ZZ{\mathbb{Z}}

\def\PP{\mathbb{P}}

\def\QQ{\mathbb{Q}}
\def\H{\mathrm{H}}

\let\sheaf=\mathcal
\let\ideal=\mathfrak
\let\tensor=\otimes
\let\into=\hookrightarrow

\newarrow{Equal}=====

\newif\ifproofmode
\proofmodetrue
\long\def\comment#1{\ifproofmode\bgroup
  \color{red} \tiny #1\egroup\fi}

\begin{document}

\title{Indecomposable and Noncrossed Product Division Algebras over Function Fields of Smooth $p$-adic Curves}


\author{E.~Brussel}
\ead{brussel@mathcs.emory.edu}
\address{Emory University, Department of Mathematics and Computer
  Science, Atlanta, GA 30322, USA}

\author{K.~McKinnie\corref{cor1}\fnref{fn2}}
\ead{kelly.mckinnie@umontana.edu}
\address{University of Montana, Department of Mathematical Sciences, Missoula, MT
  59801, USA}
 \cortext[cor1]{corresponding author}

\author{E.~Tengan\fnref{fn3}}
\ead{etengan@icmc.usp.br}
\address{Universidade de S\~ao Paulo, Inst. de Ci\^encias Matem\'aticas e de Computa\c c\~ao, S\~ao Carlos, SP 13560-970,
  Brazil}

\fntext[fn2]{Supported by National Science Foundation grant DMS-0901516}
\fntext[fn3]{Supported by FAPESP grant 2008/57214-4.}

\begin{abstract} We construct indecomposable and noncrossed product
  division algebras over function fields of connected smooth curves
  $X$ over $\ZZ_p$.  This is done by defining an index preserving
  morphism $s:\Br(K\widehat(X))' \to \Br(K(X))'$ which splits
  $\res:\Br(K(X)) \to \Br(K\widehat(X))$, where $K\widehat(X)$ is the
  completion of $K(X)$ at the special fiber, and using it to lift
  indecomposable and noncrossed product division algebras over
  $K\widehat(X)$.
\end{abstract}

\maketitle


\section{Introduction}
Let $X$ be a connected smooth projective curve over $S=\Spec\Z_p$, let
$F=K(X)$ be its function field, and let $K\widehat(X)$ denote the
completion of $K(X)$ with respect to the discrete valuation on $K(X)$
defined by the special fiber $X_0$.  We define an index-preserving
homomorphism $$\Br(K\widehat(X))'\to\Br(K(X))'$$ that splits the
restriction map $\res:\Br(K(X))'\to\Br(K\widehat(X))'$.  Here the
``prime'' denotes ``prime-to-$p$''.  The field $K\widehat(X)$ is not
unlike a power series field over a number field, and using the methods
of \cite{Brussel} and \cite{Brussel6}, we construct certain exotic
kinds of division algebras over $K\widehat(X)$, and transfer these
constructions to $K(X)$ using our homomorphism.  In particular, we
have a new construction of noncrossed product division algebras and
indecomposable division algebras of unequal period and index over the
rational function field $\Q_p(t)$ (see Theorem \ref{t6} and Corollary
\ref{c4}).  The indexes of our noncrossed product examples are as low
as $q^2$, for $q$ an odd prime not equal to $p$, and $8$.

Recall if $K$ is a field, a {\it $K$-division algebra} $D$ is a division ring that is finite-dimensional
and central over $K$.  The {\it period} of $D$ is the order of the class $[D]$ in $\Br(K)$,
and the {\it index} $\ind(D)$ is the square root of $D$'s $K$-dimension.
A {\it noncrossed product} is a $K$-division algebra whose structure is not given by a Galois 2-cocycle.
Noncrossed products were first constructed by Amitsur in \cite{Amitsur}, settling a longstanding
open problem.  Since then there have been several other constructions, including
\cite{Saltman-NCP}, \cite{JW}, \cite{Brussel}, \cite{Brussel4}, \cite{RY}, \cite{Hanke} and \cite{HS}.  
Saltman recently showed that all division algebras of prime degree over our fields
are cyclic (\cite{Saltman-cyclicity}); the indexes of our examples are all divisible by the square of a prime.

A $K$-division algebra is {\it indecomposable} if it cannot be expressed as the tensor
product of two nontrivial $K$-division algebras.  It is easy to see that all division
algebras of equal period and index are indecomposable, and that all division algebras
of composite period are decomposable, so the problem of producing an indecomposable
division algebra is only interesting when the period and index are unequal prime-powers. 
Albert constructed decomposable division algebras of unequal (2-power) period and index in the 1930's,
but indecomposable division algebras of unequal period
and index did not appear until \cite{Saltman-Indecomposable} and \cite{ART}.
Since then there have been several constructions, 
including \cite{Tignol}, \cite{JW}, \cite{J_2}, \cite{SvdB}, \cite{Karpenko},
\cite{Brussel6}, and \cite{McKinnie}.  
In \cite{BT} two of the authors proved that over the function field of a $p$-adic curve,
any division algebra of (odd) prime period $q$ not equal to $p$ and index $q^2$ is decomposable,
completing the proof that all division algebras of prime period $q$ are crossed products
over such fields (the index $q$ case is \cite{Saltman-cyclicity}).

Noncrossed products over a rational function field $K(t)$ were
constructed in \cite{Brussel4}, for any $p$-adic field $K$.  However
the construction here is much more general, and our fields constitute
a much larger class.  For example, our methods apply to fields such as
$K(X)=\Q_p(t)(\sqrt{t^3+at+b})$, where $a,b\in\Z_p$, and $p\neq 2,3$
does not divide the discriminant $4a^3+27b^2$.  For here $K(X)$ is the
function field of the elliptic curve
$X=\Proj\Z_p[x,y,z]/(y^2z-x^3-axz^2-bz^3)$ (with $t=x/z$), which is
smooth over $\ZZ_p$ by \cite{Liu}, IV.3.30 and IV.3.35.
Nevertheless, it is well known that not all finite extensions of
$\Q_p(t)$ are function fields of smooth curves over $\Z_p$, as we will
indicate; we do not consider such fields in this paper.

In his Ph.D. thesis (see \cite{Feng}), Feng Chen has constructed an
index preserving homomorphism $\Br(K\widehat (X)) \to \Br(K(X))$ over
function fields of connected smooth curves, this time over an
arbitrary complete discrete valuation ring.  Chen's approach is quite
different from ours, building on patching techniques developed by
Harbater and Hartmann \cite{HH}.  We believe that both methods are of
interest, and may in the future complement each other in the study of
division algebras over function fields of curves over complete rings.
Finally, we mention that it should be possible to transfer
Hanke-Sonn's comprehensive analysis of noncrossed products in
\cite{HS} to our situation.

We would like to thank the referee for pointing out an
  error in the initial version of this paper, and for his many
  suggestions that helped improve the exposition.

\medskip
{\bf Notation.}  Throughout this paper we let $(c)$ denote the image
of $c\in K^*$ in $\H^1(K,\mu_n)$.  In general we write $a.b$ for the
cup product of cohomology classes $a$ and $b$, unless
$a\in\H^1(K,\Q/\Z)$ and $b=(c)$, in which case for historical reasons
we write $$(a,c)=a.(c)\in\Br(K).$$


\section{Tamely ramified covers of smooth curves}
In this section we review some facts about smooth curves over complete
discrete valuation rings and tamely ramified covers of them.  

\subsection{Smooth Curves and Marks}
Let $R$ be a noetherian ring.  By a \emph{smooth curve $X$ over $R$}
we mean a scheme $X$ which is projective and smooth of relative
dimension~1 over $\Spec R$.  In particular, $X$ is flat and of finite
presentation over $\Spec R$.

By a \emph{mark} $D$ on $X$ we mean an effective \'etale-relative
Cartier divisor $D$ on $X$, that is, a closed subscheme of $X$
that is \'etale over $\Spec R$ and whose defining ideal is invertible 
as an $\sheaf{O}_X$-module.

Note that the definitions of smooth curves, effective relative Cartier
divisors, and marks are stable under arbitrary base change (see
\cite{EGAIVd} 17.3.3 (iii), \cite{EGAII} 5.5.5 (iii), and
\cite{KatzMazur} 1.1.4).

In this paper we work with smooth curves over complete discrete
valuation rings.  In the next lemma we collect some useful facts about
them.

\begin{lemma}
  Let $(R, \ideal{m}, k)$ be a complete discrete valuation ring with
  maximal ideal $\ideal{m}$, residue field $k = R/\ideal{m}$, and
  field of fractions $K = \Frac R$.  Let $X$ be a smooth curve over
  $R$ and write $X_0 \df = X \times_{\Spec R} \Spec k$ for its special
  fiber (a smooth curve over $k$).  For any effective relative Cartier
  divisor $D$ on $X$, denote its restriction to $X_0$ by $D_0 \df = D
  \times_{\Spec R} \Spec k$. \label{marklemma}
  \begin{enumerate}
    \def\theenumi{\roman{enumi}}
  \item Both $X$ and $X_0$ are regular.

  \item $X$ is connected if and only if $X_0$ is connected.
    
  \item Any effective relative Cartier divisor $D$ on $X$ is finite
    over $\Spec R$.  In particular, we may write $D = \Spec S$ where
    $S$ is a product of finite free local $R$-algebras.
  
  \item Let $D$ be an effective relative Cartier divisor $D$ on $X$.
    Then
    $$D \text{ is a mark on }X\iff
    D_0 \text{ is a mark on }X_0
    $$

  \item Let $D$ be a mark on $X$.  Then
    $$D \text{ is irreducible}\iff
    D \text{ is integral}\iff
    D \text{ is connected}
    $$
    Hence there is a 1-1 correspondence between irreducible components
    of a mark $D$ and those of $D_0$, and in particular, if $D$ is an
    integral mark, then so is $D_0$.

  \item If $D$ is an integral mark then $[K(D):K]=[k(D_0):k]$ where
    $K(D)$ and $k(D_0)$ denote the function fields of $D$ and $D_0$
    respectively.

  \item Any irreducible effective Cartier divisor on $X$ other than
    the irreducible components of $X_0$ is relative.  Moreover any
    mark $D_0$ on $X_0$ lifts to a mark $D$ on $X$.
  \end{enumerate}
\end{lemma}


\begin{proof}
  Since $X$ and $X_0$ are smooth over $\Spec R$ and $\Spec k$
  respectively, (i) follows from \cite{EGAIVd} 17.5.8~(iii).  On the
  other hand (ii) is just a special case of \cite{EGAIVd} 18.5.19.
  
  The structure map $D\to \Spec R$ is proper as the composition of the closed
  immersion $D \into X$ and the projective morphism $X \to \Spec R$, so
  the first assertion of (iii) follows from \cite{KatzMazur} 1.2.3.
  The second assertion follows from the fact that (by definition)
  finite morphisms are affine, that any finite algebra $S$ over a
  henselian ring $R$ is a product of finite local $R$-algebras (see
  \cite{Milne} I.4.2~(b)), and that a finitely generated module over a
  local ring is flat if and only if it is free (see \cite{Matsumura}
  7.10).  This proves (iii).

  To prove (iv) we may assume by (iii) that $D = \Spec S$ for some
  finite free (hence flat) local $R$-algebra $S$, and it remains to
  show that $S$ is unramified over $R$ if and only if $S\tensor_R k$
  is unramified over $k$.  This follows from \cite{EGAIVd}
  17.4.1~(a),(d) since $S$, being a local ring, is unramified over $R$
  if and only if it is unramified over $R$ at its maximal ideal
  (c.f. ibid, D\'efinition 17.3.7).

  To prove (v), first observe that if $D$ is a mark, then it is
  reduced by \cite{SGAI} I.9.2 since $R$ is a domain.  Hence a mark is
  irreducible if and only if it is integral.  Clearly if $D$ is
  irreducible then it must be connected; conversely, since $D\to \Spec
  R$ is \'etale and $R$ is normal, $D$ is also normal (\cite{SGAI}
  I.9.10), hence if $D$ is connected it must be irreducible.
  Therefore connected and irreducible components of $D$ agree, and
  since $D \to \Spec R$ is proper and $R$ is henselian the rest of (v)
  follows directly from \cite{EGAIVd} 18.5.19 (or \cite{Milne} I.4.2).

  To prove (vi), write $D = \Spec S$ for some finite free local
  $R$-algebra $S$ using (iii).  Note that $[S \tensor_R K: K] = [S
  \tensor_R k: k]$ equals the rank of $S$ over $R$, hence it is enough
  to show that $S \tensor_R K = K(D)$ and $S \tensor_R k = k(D_0)$.
  Since $S$ is \'etale over $R$, $\ideal{m}S$ is the maximal ideal of
  $S$ and $S/\ideal{m} S = S \tensor_R k = k(D_0)$ is its residue
  field; on the other hand, $S \tensor_R K \subset \Frac S$ is
  a localization of $S$ that contains $S$ and is \'etale over $K$,
  hence we must have $S \tensor_R K = \Frac S = K(D)$.

  Finally the first fact in (vii) follows from \cite{Liu} IV.3.10.
  The second assertion is then a consequence of (iv) and
  \cite{Liu} VIII.3.35 (see also \cite{EGAIVd} 21.9.11~(i) and
    21.9.12).
\end{proof}

\subsection{Tamely ramified covers}
Let $K$ be a field and $v\colon K\to \ZZ\cup \{ \infty \}$ be a
discrete valuation with residue field of characteristic $p$.  Let $L/K$ be a finite separable field extension and $L'$ be the
Galois closure of $L$ in some separable closure of $K$ containing
$L$.  Let $\{w_i\}$ be the discrete
valuations of $L'$ extending $v$ and denote by $I_i$ their inertia
groups (see \cite{Bourbaki} V.2.3 or \cite{Lang} VII.2).  Recall that
$L/K$ is said to be \emph{tamely ramified} with respect to $v$
if $p$ does not divide $|I_i|$ for all $i$.

Let $X$ be an integral smooth curve over a complete discrete
  valuation ring $(R, \ideal{m}, k)$.  By Lemma~\ref{marklemma}~(i)
$X$ is regular, hence normal, so that each irreducible effective
Weil (or Cartier) divisor $E$ defines a discrete valuation on the
function field $K(X)$ of $X$, which we will denote by $v_E$.  Now let
$D$ be a mark on $X$ and $\rho\colon Y \to X$ be a finite $(\Spec
R)$-morphism of integral smooth curves over $R$.  We say that $\rho$
is a \emph{tamely ramified cover} of the pair $(X, D)$ if it is
\'etale over $X - D$ and tamely ramified along $D$, that is, the
function field $K(Y)$ of $Y$ is a tamely ramified extension of the
function field $K(X)$ of $X$ with respect to the valuations defined by
irreducible components of $D$.  \'Etale locally, tamely ramified
covers have the following description (see \cite{Wewers} 2.3.4 and
\cite{FreitagKiehl} A.I.11): for each geometric closed point $y\colon
\Spec \Omega\to Y$ with image $x = \rho \circ y\colon \Spec \Omega \to
X$ there exist affine \'etale neighborhoods $\Spec B \to Y$ and $\Spec
A \to X$ of $y$ and $x$ such that $B = A[w]/(w^n - z)$ for some $z \in
A$ (an \'etale local coordinate of $D$) and some integer $n$ prime to
the characteristic of $k$.

\begin{lemma} Let $X$ be an integral smooth curve over a complete
  discrete valuation ring $(R, \ideal{m}, k)$, and $D$ be a mark on
  $X$.  Let $\rho\colon Y \to X$ be a tamely ramified cover of $(X,
  D)$.  Let $E$ be a mark on $X$ such that either $E \cap D =
  \emptyset$ or $E \subset D$.  Then \label{specialval}
  \begin{enumerate}
    \def\theenumi{\roman{enumi}}

  \item $Y$ is flat over $X$ and equals the normalization of $X$ in
    $K(Y)$;

  \item $(\rho^{-1} E)_{\rm red}$ is a mark;

  \item if $E$ is irreducible and $F$ is an
    irreducible mark on $Y$ lying over $E$, then the
    ramification (resp. the inertia) degree of $v_F$ over $v_E$ equals
    the ramification (resp. inertia) degree of $v_{F_0}$ over
    $v_{E_0}$.
  \end{enumerate}
\end{lemma}

\begin{proof} The restriction $\rho_0\colon Y_0 \to X_0$ of $\rho$ to the
  special fibers is a finite generically \'etale map between smooth
  curves over a field, which is flat by \cite{SGAI} IV.1.3~(ii) for
  instance.  Hence $\rho$ is also flat by the local criterion of
  flatness (see \cite{SGAI} IV.5.9).  To finish the proof of (i), note
  that $Y$ is regular (Lemma~\ref{marklemma}~(i)) and thus normal, and
  since it is also integral over $X$, it equals the normalization of
  $X$ in $K(Y)$.

  To prove (ii), assume first that $E \cap D = \emptyset$.  Since
  $\rho^{-1} (X-D) \to X - D$ is \'etale by assumption, 
  $\rho^{-1}E \to E$ is \'etale by base change.  Therefore since $E$ is reduced,
  $\rho^{-1}E$ is reduced (\cite{SGAI} I.9.2), and since $E\to\Spec R$ is already \'etale, the
  composition $\rho^{-1}E \to E\to \Spec R$ is \'etale, hence
  $(\rho^{-1} E)_{\rm red} = \rho^{-1}E$ is a mark.

  Now suppose that $E\subset D$; we may assume without loss of
  generality that $E$ is a connected and hence irreducible
  component of $D$ (see Lemma~\ref{marklemma}~(v)).  We first
    show that each connected component of $\rho^{-1}E$ is irreducible.
    We start by understanding the situation locally.

Let
  $x_0$ be the closed point of $E$ and $y_0\in Y_0$ be such that
  $\rho(y_0) = x_0$.  Write $A = \sheaf{O}_{X, x_0}$ and $B =
  \sheaf{O}_{Y, y_0}$; both are 2-dimensional noetherian
  regular local rings and hence also factorial domains by Auslander-Buchsbaum's theorem.  We
  have an exact sequence
  $$0\rTo \sheaf{I}_{E, x_0} \rTo \sheaf{O}_{X, x_0} \rTo
  \sheaf{O}_{E, x_0} \rTo 0
  $$
  which can be rewritten as
  $$0 \rTo A \rTo^z A \rTo A/(z) \rTo 0$$
  where $z\in A$ is a prime element defining $E$ so that $E = \Spec
  A/(z)$ (recall that $E$ is a local affine scheme by
  Lemma~\ref{marklemma}~(iii)).  
Since $A/(z)$ is \'etale over the complete discrete valuation ring $R$, 
$A/(z)$ is normal, hence it is a complete discrete valuation ring,
and it follows that $z$ is part of a regular system of parameters of $A$, and that $(z)+\frak m A$ is the
maximal ideal of $A$.
Write $A_0 \df= \sheaf{O}_{X_0, x_0}=A\otimes_R k
  = A/\ideal{m}A$, and let $z_0$ be the image of $z$ in $A_0$.
  Then $A_0=(A_0,(z_0),k)$ is a discrete valuation ring whose maximal ideal $(z_0)$ defines $x_0=E_0$.

The closed subscheme $\rho^{-1}E$ of $Y$ is a relative Cartier
divisor by flat pull-back (see \cite{KatzMazur}, 1.1.4), hence by
Lemma~\ref{marklemma}~(iii) there is a 1-1 correspondence between the
connected components of $\rho^{-1}E$ and its closed points.
Showing that the connected component of $\rho^{-1}E$ going through
  $y_0$ is irreducible is now equivalent to
  showing that $z$ is divisible by a single prime factor $w$ in $B$,
  so that it can be written as $z = u \cdot w^e$, $u\in B^\unit$.
  This follows from the \'etale local description of $\rho\colon Y \to
  X$ (see \cite{Wewers} 2.3.4 and \cite{FreitagKiehl}, I.3.2
    and A.I.11): since $A$ is a regular local ring, $B$ is the
    localization of the normalization of $A$ in $K(Y)$ with respect to
    one of its maximal ideals, and $B$ is tamely ramified just along
    $z$, for some integer $e$ prime to $\chara k$ we have a commutative diagram 
  \begin{diagram}
    B&\rTo & B_{sh} \rlap{${}= \dfrac{A_{sh}[T]}{(T^e - z)}$}\\
    \uTo&&\uTo\\
    A&\rTo& A_{sh}
  \end{diagram}
where $A_{sh}$ denotes the strict henselization of $A$, and $B_{sh}$ that of $B$.
Observe that all four maps are faithfully flat (see
    \cite{EGAIVd} 18.8.8~(iii)), and that all four rings are noetherian
    regular local rings, and thus factorial domains.  In particular, the associated
    maps between spectra are surjective (see \cite{Milne} I.2.7~(c))
    and preserve height~1 prime ideals (by \cite{Matsumura} 15.1), all
    of which are principal (ibid. 20.1).  Hence in order to show that
  there is a single prime $(w)$ in $B$ lying over $(z)$, it is enough
  to show that there is a single prime in $B_{sh}$ lying over $(z)$.
  Notice that $z$ stays prime in $A_{sh}$.  For as shown above, $z$ is part of a regular system of parameters of $A$.  Since $A_{sh}$ is
  unramified over $A$, the maximal ideal of $A_{sh}$ is
    generated by the maximal ideal of $A$, and therefore $z$ is also
  part of a regular system of parameters of $A_{sh}$, which implies
  that $z$ is prime in $A_{sh}$.
The only prime in $B_{sh}$
  lying above $(z)$ is $(\overline T)$, since the fiber $\Spec
  B_{sh}\tensor_{A_{sh}} \kappa(z) = \Spec \kappa(z)[T]/(T^e)$ consists of a
  single point (here $\kappa(z) = \Frac A/(z)$ denotes the residue
  field of the prime $(z)$).  The image of $(\overline T)$ in $\Spec
  B$ is the unique prime $(w)$ lying over $(z)$, and since $B_{sh}$ is unramified over $B$,
   $(w)B_{sh} = (\overline T)$ in $B_{sh}$.
This completes the proof that each connected component of $\rho^{-1}E$ is irreducible.

It remains to show that each connected (irreducible) component of $(\rho^{-1}E)_{\rm red}$ is a mark.
  Write $B_0 \df = \sheaf{O}_{Y_0, y_0} = B \tensor_R k =
  B/\ideal{m}B$ and let $w_0$ be the image of $w$ in $B_0$.  By Lemma
  \ref{marklemma}~(iv), to show that the connected component $\Spec
  B/(w)$ of $(\rho^{-1}E)_{\rm red}$ is a mark, it is enough to show
  that its restriction $\Spec B_0/(w_0)$ to the special fiber is a
  mark, i.e., that $w_0$ is a uniformizer of the discrete valuation
  ring $B_0$.  Now observe that the extension of the ideal $\ideal{b} \df = (w) +
  \ideal{m}B$ of $B$ to $B_{sh}$ is the maximal ideal $\ideal{b} B_{sh} = (\overline T) +
  \ideal{m}B_{sh}$ of $B_{sh}$: in fact, by direct computation we have
  an isomorphism
  $$\dfrac{B_{sh}}{(\overline T) + \ideal{m}B_{sh}} =
  \dfrac{A_{sh}}{(z) + \ideal{m}A_{sh}}
  $$
  and since $(z) + \ideal{m}A$ is the maximal ideal of $A$, $(z) +
  \ideal{m}A_{sh}$ is the maximal ideal of $A_{sh}$.  On the other
  hand, $B_{sh}$ is faithfully flat over $B$, hence $\ideal{b}$ must
  be the maximal ideal of $B$ (see \cite{Matsumura} 7.5~(ii)).
  Therefore $(w_0)$, the image of $\ideal{b}$ in $B_0$, is the maximal
  ideal of $B_0$, as was to be shown.

  We now prove (iii).  Note that $E_0$ and $F_0$ are irreducible marks
  by Lemma \ref{marklemma}(v) so that $v_{E_0}$ and $v_{F_0}$ are
  well-defined.  Denoting $K = \Frac R$, and by $K(F)$, $K(E)$,
  $k(F_0)$, $k(E_0)$ the function fields of $F$, $E$, $F_0$, $E_0$, we
  have by Lemma~\ref{marklemma}~(vi) that
  $$[K(F): K(E)] = \frac{[K(F) : K]}{[K(E): K]}
                 = \frac{[k(F_0) : k]}{[k(E_0): k]}
                 = [k(F_0): k(E_0)]$$
  showing that the inertia degree of $v_F$ over $v_E$ equals that of
  $v_{F_0}$ over $v_{E_0}$.

  To show equality of ramification degrees, we keep the notation in
  the proof of (ii).  If $E\cap D = \emptyset$, then $F\to E$ and
  $F_0\to E_0$ are both \'etale, so the ramification degree is 1 in
  both cases.  Now assume that $E \subset D$; observe that $z$ and $w$
  are uniformizers of $v_E$ and $v_F$, respectively.  We showed above
  that $(z_0)$ is the maximal ideal of the discrete valuation ring
  $A_0$, i.e., $z_0$ is a uniformizer of $v_{E_0}$, and similarly
  $w_0$ is a uniformizer of $v_{F_0}$.  Since $z=w^e \cdot u$, where
  $u\in B^\unit$ and $e$ is the ramification degree of $v_F$ over
  $v_E$, we have $z_0 = w_0^e \cdot u_0$, where $u_0$ is the image of $u$ in $B_0^\unit$, 
  and thus $v_{F_0}(z_0) = e$ as desired.
\end{proof}

\subsection{An equivalence of categories}
Let $(R, \ideal{m}, k)$ be a complete discrete valuation ring,
$X$ be a smooth integral curve over $R$, and $D$ be a mark on $X$.  We
write $\Rev^D_R(X)$ for the category whose objects are the tamely
ramified covers of $(X, D)$ and whose arrows are the $X$-morphisms.
We have a restriction functor $\Rev^D_R(X) \to
\Rev^{D_0}_k(X_0)$ taking a tamely ramified cover $Y$ of $(X, D)$ to
the tamely ramified cover $Y_0$ of $(X_0, D_0)$, 
and a map $f\colon Y \to Z$ to its restriction $f_0 \df = f
\times_{\Spec R} \Spec k\colon Y_0 \to Z_0$ to the special fibers.
Observe that by Lemma~\ref{marklemma}~(i) and the definition of tamely
ramified cover all objects in $\Rev^D_R(X)$ and $\Rev^{D_0}_k(X_0)$
are regular schemes.

Amazingly, this functor $\Rev^D_R(X) \to \Rev^{D_0}_k(X_0)$ is an
equivalence of categories (see \cite{Wewers} 3.1.3 for the proof):

\begin{theorem} \label{equivcat} (Grothendieck) Let $(R, \ideal{m},
  k)$ be a complete discrete valuation ring, $X$ be a smooth
  integral curve over $R$, and $D$ be a mark on $X$.  Then restriction
  to the special fibers gives an equivalence of categories
  $$\Rev^D_R(X) \buildrel \approx \over \to \Rev^{D_0}_k(X_0)$$
\end{theorem}

For any scheme $X$ and effective Cartier divisor $D$ we write
$\pi_1^t(X, D, \overline x)$ for the \emph{tame fundamental group of
$X$ with respect to $D$} with geometric base point $\overline x\colon \Spec \Omega
\to X-D$ (see \cite{Wewers} 4.1.2, \cite{SGAI} XIII.2.1.3 or
\cite{FreitagKiehl} A.I.13).  By definition, $\pi_1^t(X, D, \overline
x)$ classifies pointed tamely ramified covers of $(X, D)$, and thus we
obtain the following (c.f. \cite{SGAI} X.2.1)

\begin{corollary} With the notation and hypotheses of the previous
  theorem, let $\overline x_0$ be a geometric point of $X_0 - D_0$.
  Then the natural map \label{tamepicorollary}
  $$\pi_1^t(X_0, D_0, \overline x_0) \buildrel \approx \over\to
    \pi_1^t(X, D, \overline x_0)$$
  of tame fundamental groups is an isomorphism.
\end{corollary}


\subsection{The residue map}
In what follows, all cohomology groups are \'etale cohomology groups.
For a ring $R$ and \'etale sheaf $F$ on $\Spec R$ we write $\H^a(R,
F)$ instead of $\H^a(\Spec R, F)$.  In particular, for a field $K$,
$\H^a(K, F)$ agrees with the Galois cohomology group $\H^a(G_K, F)$
where $G_K = \Gal(K_{\rm sep}/K)$ denotes the absolute Galois group of
$K$ and where we still write $F$ for the corresponding $G_K$-module.

Let $K$ be any field, let $v\colon K \to \ZZ \cup \{ \infty \}$ be a
discrete valuation on $K$, and let $k$ be its residue field.  Recall
that for any integer $r$ and any integer $n$ prime to the
characteristic of $k$ there is a group morphism
$$\partial_v\colon \H^a(K, \mu_n^{\tensor r}) \to
               \H^{a-1}(k,\mu_n^{\tensor (r-1)})
$$
called the \emph{residue} or \emph{ramification map} (see \cite{GMS}
II.7.9 or \cite{GilleSzamuely} VI.8).  The residue map has the
following functorial behavior: if $L$ is a finite extension of $K$
and $w\colon L \to \ZZ \cup \{ \infty \}$ is a discrete valuation with
residue field $l$ such that $w$ extends $v$ then we have a commutative
diagram
\begin{diagram}
\H^a(L, \mu_n^{\tensor r}) &\rTo^{\partial_w} & \H^{a-1}(l, \mu_n^{\tensor (r-1)})\\
\uTo^{\res}&&\uTo_{e_{w/v} \cdot \res}\\
\H^a(K, \mu_n^{\tensor r}) &\rTo^{\partial_v}& \H^{a-1}(k, \mu_n^{\tensor (r-1)})
\end{diagram}
where $e_{w/v}$ denotes the ramification degree of $w$ over $v$, and $\res$ denotes
cohomological restriction.

If $X$ is a normal integral scheme and $D\subset X$ is an irreducible
Weil divisor then we write 
$$\partial_D\colon \H^a(K(X), \mu_n^{\tensor r}) \to \H^{a-1}(K(D), \mu_n^{\tensor (r-1)})$$
for the residue map with respect to the discrete valuation $v_D$.

\begin{lemma} Let $X$ be a smooth curve over a complete
  discrete valuation ring, and let $n$ be an invertible
  integer on $X$ (i.e., $n$ is prime to all residue characteristics on
  $X$).  Let $D$ be a mark on $X$, $U = X - D$, and denote by $j\colon
  U\into X$ and $i\colon D \into X$ the corresponding open and closed
  immersions.  We have an exact Gysin sequence \label{gysinlemma}
  \begin{eqnarray*}
    0 &\to& \H^1(X, \mu^{\tensor r}_n) \to \H^1(U, \mu^{\tensor r}_n)
    \to \H^0(D, \mu^{\tensor (r-1)}_n)\\
    &\to& \H^2(X, \mu^{\tensor r}_n) \to \H^2(U, \mu^{\tensor r}_n) \to
    \H^1(D, \mu^{\tensor (r-1)}_n)\\
    &\to& \H^3(X, \mu^{\tensor r}_n) \to \H^3(U, \mu^{\tensor r}_n) \to
    \H^2(D, \mu^{\tensor (r-1)}_n)\to \cdots
  \end{eqnarray*}  
  where $\H^a(X, \mu^{\tensor r}_n) \to \H^a(U, \mu^{\tensor r}_n)$ are
  the natural restriction maps, and the maps $\H^a(U, \mu^{\tensor
    r}_n) \to \H^{a-1}(D, \mu^{\tensor (r-1)}_n)$ are compatible with
  the residue maps in the sense that the following diagram
  commutes up to sign:
  \begin{diagram}
    \H^a(U, \mu^{\tensor r}_n) &\rTo& \H^{a-1}(D, \mu^{\tensor
      (r-1)}_n)\\
    \dTo && \dTo\\
    \H^a(K(X), \mu^{\tensor r}_n) &\rTo^{\partial_D}& \H^{a-1}(K(D), \mu^{\tensor
      (r-1)}_n)
  \end{diagram}
 \end{lemma}

 \begin{proof} Note the conclusions make sense even if $D$ is
   reducible, for in this case $D$ is the disjoint union of its
   irreducible components and $K(D)$ is a direct product of the
   corresponding function fields.  The long exact Gysin sequence will
   follow once we show that
   $$R^qj_* \mu_{n, U}^{\tensor r} = \begin{cases}
     \mu_{n, X}^{\tensor r}& \mbox{if } q=0\\
     i_* \mu_{n, D}^{\tensor (r-1)}& \mbox{if } q=1\\
     0& \mbox{if } q\ge 2
  \end{cases}$$ 
For then the Leray spectral sequence 
$$\H^p(X, R^qj_* \mu_{n, U}^{\tensor r}) \Longrightarrow
    \H^{p+q}(U, \mu_{n, U}^{\tensor r})
$$
degenerates, and as $i_*$ is an exact functor we may substitute
$H^{q-1}(D, \mu_{n, D}^{\otimes(r-1)})$ for $H^{q-1}(X, i_*\mu_{n,
  D}^{\otimes (r-1)})$, by the Leray spectral sequence for $i_*$.

Since $D$ is a mark, $(X, D)$ is a smooth $(\Spec R)$-pair of
codimension $c=1$, and hence by purity (\cite{Milne} VI.5.1) we
already know that $R^qj_* \mu_{n, U}^{\tensor r} = 0$ for $q\ne 0, 1$,
and that $j_*\mu_{n, U}^{\tensor r} = \mu_{n, X}^{\tensor r}$.  It
remains to compute $R^1j_* \mu_{n, U}^{\tensor r}$.

By \cite{SGAIVc} XIX.3.3 we know that $R^1 j_* \mu_{n, U} = i_*
\bigl((\ZZ/n)_D\bigr)$.  For the general case, consider the cup
product map
$$\mu_{n, X}^{\tensor (r-1)} \tensor i_*\bigl((\ZZ/n)_D\bigr)
  = R^0 j_* \mu_{n, U}^{\tensor (r-1)} \tensor R^1 j_* \mu_{n, U} \rTo^{\cup}
    R^1 j_* \mu_{n, U}^{\tensor r}
$$
We see this is an isomorphism by looking at stalks.  Since $i^*\mu_{n,
  X}^{\tensor (r-1)} = \mu_{n, D}^{\tensor (r-1)}$, we obtain a
sequence of maps

$$R^1 j_* \mu_{n, U}^{\tensor r}
   \lTo^\cup_\approx \mu_{n, X}^{\tensor (r-1)} \tensor i_*\bigl((\ZZ/n)_D\bigr)
   \rTo^{can} i_* i^*\mu_{n, X}^{\tensor (r-1)} \tensor i_*\bigl((\ZZ/n)_D\bigr)
   \rTo^\cup i_*\mu_{n, D}^{\tensor (r-1)}
$$
which we see are isomorphisms, again by looking at stalks.  This
yields the required isomorphism $R^1j_* \mu_{n, U}^{\tensor r} = i_*
\mu_{n, D}^{\tensor (r-1)}$.

  Finally, to prove the compatibility with the residue map, we
  may assume that $D$ is connected.  Observe that $K(D)$ is the residue field of
  $\sheaf{O}_{v_D}$.  By the naturality of the Leray spectral sequence we have a
  commutative diagram
\[\xymatrix@C=18pt{
  \cdots\ar[r]& \H^a(X, \mu^{\tensor r}_n)\ar[r]\ar[d]&\H^a(U, \mu^{\tensor r}_n)\ar[r]\ar[d]&\H^{a-1}(D, \mu^{\tensor(r-1)}_n)\ar[r]\ar[d]&\cdots\\
  \cdots\ar[r]& \H^a(\sheaf{O}_{v_D}, \mu^{\tensor r}_n)\ar[r]&\H^a(K(X), \mu^{\tensor r}_n)\ar[r]^(.425){(*)}& \H^{a-1}(K(D), \mu^{\tensor(r-1)}_n)\ar[r]&\cdots
  }\]
  whose rows are Gysin sequences, and $(*)$ is known to be the
  residue map with respect to the valuation $v_D$ (see \cite{CT}
  \S3.3) possibly up to sign.
\end{proof}

\begin{remark} \label{Faddeev} Let $K$ be a field.  We apply the
  previous lemma to $X = \Spec K[t]$.  Since $X \to \Spec K$ is
  acyclic (\cite{Milne} VI.4.20) we have $\H^a(X, \mu^{\tensor
    r}_n) = \H^a(K, \mu^{\tensor r}_n)$.  Moreover, any mark $D$ is a disjoint
  union of closed points, therefore we have $\H^{a-1}(D, \mu^{\tensor (r-1)}_n) =
  \bigoplus_{P\in D} \H^{a-1}(K(P), \mu^{\tensor (r-1)}_n)$.  Thus the
  Gysin sequence for $X$ reads
  $$\cdots \to \H^a(K, \mu^{\tensor r}_n) \to \H^a(U, \mu^{\tensor r}_n)
  \to \bigoplus_{P\in D}  \H^{a-1}(K(P), \mu^{\tensor (r-1)}_n)\to \cdots
  $$
  where $U = X - D$.  On the other hand, since \'etale cohomology
  commutes with projective limits of schemes (\cite{Milne} III.1.16)
  and $\Spec K(t) = \projlim_{D} X - D$, where $D$ runs over all marks
  of $X$, by taking limits we obtain
  $$\cdots \to \H^a(K, \mu^{\tensor r}_n) \to \H^a(K(t), \mu^{\tensor r}_n)
  \to \bigoplus_{P\in X^{(1)}}  \H^{a-1}(K(P), \mu^{\tensor (r-1)}_n)\to \cdots
  $$
  where $X^{(1)}$ denotes the set of closed points (i.e. points of
  codimension 1) of $X$.  This is just the familiar (affine) Faddeev
  sequence with finite coefficients (\cite{GilleSzamuely} 6.9.3),
  which splits into short exact sequences
  $$0 \to \H^a(K, \mu^{\tensor r}_n) \to \H^a(K(t), \mu^{\tensor r}_n)
    \to \bigoplus_{P\in X^{(1)}}  \H^{a-1}(K(P), \mu^{\tensor (r-1)}_n)\to
    0
  $$
  via the \emph{coresidue maps}
  \begin{align*}
    \psi_P \colon \H^{a-1}(K(P), \mu^{\tensor (r-1)}_n) &\to
                  \H^{a}(K(t), \mu^{\tensor r}_n)\\
    \xi &\mapsto \cor_{K(P)(t)|K(t)}(\xi.(t-\tau_P))
  \end{align*}
  where $\tau_P$ denotes the image of $t$ in $K(P)$ (so that $K(P) =
  K(\tau_P)$) and $(t-\tau_P)$ is the image of $t-\tau_P$ in $ \H^1(K(P)(t),\mu_n)$.
  
\end{remark}


\section{Splitting the restriction map}

\subsection{Setup and conventions}
Henceforth we write \label{mainsetup}
\begin{itemize}
\item $(R, \ideal{m}, k) = \hbox{complete}$ discrete valuation ring
  with finite residue field $k$ of characteristic $p$, and fraction
  field $K=\Frac R$ (a local field);

\item $\pi = \hbox{a uniformizer}$ of $R$;

\item $n = \hbox{integer}$ prime to $p$;

\item $X = \hbox{a}$ smooth integral curve over $R$;

\item $X_0 =
  \hbox{the}$ special fiber of $X$ (a smooth integral curve over $k$);

\item $K(X) = \hbox{the}$ function field of $X$.

\item $k(X_0) =
  \hbox{the}$ function field of $X_0$ (a global field);

\item $K\widehat(X) = \hbox{completion}$ of $K(X)$ with respect to the
  valuation defined by the special fiber $X_0$.  Observe that $\pi$ is
  also a uniformizer of $K\widehat(X)$ and that its residue field is
  $k(X_0)$;

\item $V = \hbox{a fixed set}$ of marks on $X$ lifting each mark (i.e
  closed point) of $X_0$, see Lemma~\ref{marklemma}~(vii).
\end{itemize}

By \cite{Liu} VIII.3.4, the set $V$ is in 1-1 correspondence with a
subset of closed points of the generic fiber $X_{\eta} \df= X
\times_{\Spec R} \Spec K$.  In what follows, we will identify these
two sets and refer to the unique mark $D \in V$ (or closed point $P\in
X_{\eta}$ whose closure equals $D$) lifting a closed point $P_0\in
X_0$ as the \emph{$V$-lift} of $P_0$.  For instance, if $X = \PP^1_R =
\Proj R[x, y]$ and we choose the mark defined by $y$ to be the
$V$-lift of the ``infinite point" of $X_0 = \PP^1_k = \Proj k[x,
y]$ defined by $y$, then specifying the remaining $V$-lifts amounts to
choosing a monic lift in $R[t]$ for each monic irreducible polynomial
in $k[t]$ (where $t = x/y$).


\subsection{Splitting the restriction map}
\label{splitting}
In this section we construct a map
$$s = s_{V, \pi} \colon \Br(K\widehat(X))' \to \Br(K(X))'$$
splitting the restriction map
$$\res\colon \Br(K(X))' \to \Br(K\widehat(X))'$$
Here ${}'$ denotes the prime-to-$p$ part of the corresponding group.
In the next section we show that this map preserves the index.

\begin{lemma} \label{tamelift} (Tame lifting) The choice of $V$
  defines, for each $a\ge 0$ and $r\in \ZZ$, a group morphism
$$\lambda_V\colon \H^a(k(X_0), \mu_n^{\tensor r}) \to \H^a(K(X), \mu_n^{\tensor
  r})$$
compatible with the residue maps: for each irreducible mark $D \in V$,
\begin{diagram}
\H^a(k(X_0), \mu_n^{\tensor r}) &\rTo^{\lambda_V}& \H^a(K(X), \mu_n^{\tensor r})\\
\dTo^{\partial_{D_0}} &&\dTo_{\partial_D}\\
\H^{a-1}(k(D_0), \mu_n^{\tensor (r-1)}) &\rTo & \H^{a-1}(K(D), \mu_n^{\tensor (r-1)})
\end{diagram}
commutes up to sign, where the bottom arrow is given by the composition
$$\H^{a-1}(k(D_0), \mu_n^{\tensor (r-1)}) \lTo^{can}_{\approx}
  \H^{a-1}(D, \mu_n^{\tensor (r-1)}) \rTo^{can} \H^{a-1}(K(D),  \mu_n^{\tensor (r-1)}).
$$
\end{lemma}

\begin{proof} Let $D$ be a mark with support in $V$, and set $U = X -
  D$.  Consider the commutative diagram
  \begin{diagram}
    \cdots & \rTo & \H^a(X, \mu^{\tensor r}_n) &\rTo& \H^a(U, \mu^{\tensor r}_n)
           & \rTo & \H^{a-1}(D, \mu^{\tensor (r-1)}_n) &\rTo& \cdots\\
           &&\dTo^{\approx} && \dTo && \dTo^{\approx}\\
    \cdots & \rTo & \H^a(X_0, \mu^{\tensor r}_n) &\rTo& \H^a(U_0, \mu^{\tensor r}_n)
           & \rTo & \H^{a-1}(D_0, \mu^{\tensor (r-1)}_n) &\rTo& \cdots\\
  \end{diagram}
  where the rows are the exact Gysin sequences for $(X, D)$ and $(X_0, D_0)$
  respectively (see Lemma~\ref{gysinlemma}), and the vertical arrows are the natural
  ones (restrictions to the fibers).  Since $R$ is henselian, the left and right arrows are
  isomorphisms by proper base change (\cite{Milne} VI.2.7), hence so is
  the middle one by the 5-lemma.

  Now define $\lambda_D$ as the composition
  $$\lambda_D\colon \H^a(U_0, \mu^{\tensor r}_n) \lTo^{\approx} \H^a(U,
  \mu^{\tensor r}_n) \rTo^{can} \H^a(K(X), \mu_n^{\tensor r})$$
  Consider the set $\mathcal{V}$ of all marks with support in $V$ and order
  them by inclusion.  Since \'etale cohomology commutes with
  projective limits of schemes (\cite{Milne} III.1.16) and
  $$\Spec k(X_0) = \projlim_{D\in \mathcal{V}} U_0$$
  taking the direct limit of the $\lambda_D$ over all $D \in
  \mathcal{V}$ we obtain the desired map $\lambda_V\colon \H^a(k(X_0),
  \mu^{\tensor r}_n) \to \H^a(K(X), \mu_n^{\tensor r})$.  Since by
  Lemma~\ref{gysinlemma} the Gysin sequences are compatible with
  residue maps up to sign, and the arrow 
  \[\H^{a-1}(k(D_0), \mu_n^{\tensor (r-1)}) \lTo^{can}_{\approx}
  \H^{a-1}(D, \mu_n^{\tensor (r-1)})\]
  is invertible, we see that $\lambda_V$ is also compatible
  with residue maps.
\end{proof}

\begin{remark}
  In case $X = \PP^1_R$, we can give a more explicit description of
  the tame lifting using the Faddeev sequence (see Remark
  \ref{Faddeev}).  Lifting the point at infinity as in the example of
  Section~\ref{mainsetup}, the map $\lambda_V$ can be defined by the
  following commutative diagram
  \[\xymatrix@C=14pt{
  0\ar[r]&\H^a(k, \mu_n^{\tensor r})\ar[r]\ar[d]&\H^a(k(X_0), \mu_n^{\tensor r})\ar[r]\ar@{-->}[d]^{\lambda_V}&\lower5ex\hbox{${\bigoplus_{P_0 \in X_0^{(1)}}\limits} \H^{a-1}(k(P_0), \mu_n^{\tensor (r-1)})$}\ar[r]\ar[d]&0\\
  0\ar[r]& \H^a(K, \mu_n^{\tensor r})\ar[r]&\H^a(K(X), \mu_n^{\tensor r})\ar[r]&\lower5ex\hbox{${\bigoplus_{ P\in X_\eta^{(1)}}\limits}$}\lower5ex\hbox{$H^{a-1}(K(P), \mu_n^{\tensor (r-1)})$}\ar[r]&0
  }\]

  \noindent where each row is the split exact Faddeev sequence of
  Remark \ref{Faddeev}.  The left vertical arrow is the natural one
  while the right vertical arrow sends, via the natural map
  $\H^{a-1}(k(P_0), \mu_n^{\tensor (r-1)}) \to \H^{a-1}(K(P),
  \mu_n^{\tensor (r-1)})$, the $P_0$-th component to the
  $P$-th component where $P$ denotes the generic point of the $V$-lift
  of $P_0$.  Explicitly, using the splitting given by the coresidue
  maps $\psi_{P_0}$, we may write an element of $\H^a(k(X_0),
  \mu_n^{\tensor r})$ as $\alpha_0 + \sum_{P_0} \psi_{P_0}(\xi_{P_0})$
  with $\alpha_0\in \H^a(k, \mu_n^{\tensor r})$ and $\xi_{P_0} \in
  \H^{a-1}(k(P_0), \mu_n^{\tensor (r-1)})$.  Then
  $$\lambda_V \Bigl(\alpha_0 + \sum_{P_0} \psi_{P_0}(\xi_{P_0})\Bigr)
  = \alpha + \sum_{P} \psi_P(\xi_P)
  $$
  where $P$ is the closed point of $X_\eta$ corresponding to the
  $V$-lift of $P_0$ and $\alpha \in \H^a(K, \mu_n^{\tensor r})$ and
  $\xi_P \in \H^{a-1}(K(P), \mu_n^{\tensor (r-1)})$ denote the
  unramified lifts of $\alpha_0$ and $\xi_{P_0}$ respectively.
\end{remark}

\begin{lemma} \label{uglylemma} Let $\chi_0\in \H^1(k(X_0), \ZZ/n)$, and
  let $D_0 \subset X_0$ be the ramification locus of $\chi_0$.  Denote
  by $Y_0$ the cyclic tamely ramified cover of $(X_0, D_0)$ defined by
  $\chi_0$.  Let $\chi = \lambda_V(\chi_0) \in \H^1(K(X), \ZZ/n)$ be as
  in the previous lemma.  Then $\chi$ defines the tamely ramified
  cover $Y$ of $(X, D)$ lifting $Y_0$ in Theorem \ref{equivcat}, where $D$ is the $V$-lift of
  $D_0$.
\end{lemma}

\begin{proof}
  By definition of $\lambda_V$, $\chi \in \H^1(X-D, \ZZ/n) \subset
  \H^1(K(X), \ZZ/n)$ is the unique character that restricts to $\chi_0
  \in \H^1(X_0-D_0, \ZZ/n) \subset \H^1(k(X_0), \ZZ/n)$.  Since the
  groups $\H^1(X-D, \ZZ/n) = \Hom_{\rm cont}(\pi_1^t(X, D), \ZZ/n)$
  and $\H^1(X_0-D_0, \ZZ/n) = \Hom_{\rm cont}(\pi_1^t(X_0, D_0),
  \ZZ/n)$ classify degree $n$ (tame) cyclic Galois covers of $(X, D)$
  and $(X_0, D_0)$ (see \cite{FreitagKiehl} I.2.11), and the
  restriction map $\res\colon \H^1(X-D, \ZZ/n) \to \H^1(X_0-D_0,
  \ZZ/n)$ is given by the natural map $\pi_1^t(X_0, D_0) \rTo^\approx
  \pi_1^t(X, D)$ induced by the functor $Y \mapsto Y_0$ (see Corollary
  \ref{tamepicorollary}), the cyclic Galois cover $Y$ of $(X, D)$
  defined by $\chi$ restricts to the cyclic Galois cover $Y_0$ of
  $(X_0, D_0)$ defined by $\chi_0$, and we are done.\end{proof}

\begin{theorem} Let $X$, $K(X)$, $K\widehat(X)$ and $n$ be as in Section \ref{mainsetup}.  Each choice of $\pi$ and $V$ as in Section \ref{mainsetup} defines, for each $a\ge
  0$ and all $r\in \ZZ$, a group morphism
  $$  s = s_{V, \pi}\colon \H^a(K\widehat(X), \mu_n^{\tensor r}) \to
                          \H^a(K(X), \mu_n^{\tensor r})
  $$
  splitting $\res\colon \H^a(K(X), \mu_n^{\tensor r}) \to \H^a(\hat
  K(X), \mu_n^{\tensor r})$, that is, such that $\res\circ s$ is the identity.
\end{theorem}

\begin{proof}
  Let $A = \sheaf{O}_{X, \eta_0}$ where $\eta_0$ denotes the generic
  point of $X_0 \subset X$.  Then $A$ is a discrete valuation
  ring; let $\hat A$ be its completion, so that $K\widehat(X) = \Frac
  \hat A$.  Observe that the residue fields of both $A$ and $\hat A$
  are equal to $k(X_0)$, and that $\pi$ is a uniformizer for both
  discrete valuation rings.  We have an exact Witt sequence (see
  \cite{GMS} II.7.10 and II.7.11)
  \[\xymatrix@C=17pt{
  0\ar[r]&\H^a(k(X_0), \mu_n^{\tensor r}) \ar[r]&
           \H^a(K\widehat(X), \mu_n^{\tensor  r})  \ar[r]^(.45){\partial_{X_0}}&
           \H^{a-1}(k(X_0), \mu_n^{\tensor (r-1)})  \ar[r]& 0
}\]
  split by the cup product with $(\pi) \in \H^1(K\widehat(X), \mu_n)$:
  $$\H^{a-1}(k(X_0), \mu_n^{\tensor (r-1)}) \rTo^{- \,.\, (\pi)}
    \H^a(K\widehat(X), \mu_n^{\tensor r})
  $$
  Hence each element of $\H^a(K\widehat(X), \mu_n^{\tensor r})$ can be
  uniquely written as a sum $\alpha_0 +\chi_0.(\pi)$ with
  \begin{equation*}
  \begin{split}
    \alpha_0&\in \H^a(k(X_0), \mu_n^{\tensor r}) = \H^a(\hat A,
    \mu_n^{\tensor r}) \subset \H^a(K\widehat(X), \mu_n^{\tensor r})\qquad \text{and}\\
    \chi_0 &\in \H^{a-1}(k(X_0), \mu_n^{\tensor (r-1)}) = \H^{a-1}(\hat
    A,
    \mu_n^{\tensor (r-1)}) \subset \H^{a-1}(K\widehat(X), \mu_n^{\tensor (r-1)})
    \end{split}
  \end{equation*}
  We define
  $$s( \alpha_0 + \chi_0.(\pi) ) =
       \alpha + \chi.(\pi)
  $$
  where
  \begin{eqnarray*}
    \alpha &=& \lambda_V(\alpha_0) \in \H^a(K(X), \mu_n^{\tensor r})
    \qquad \text{and}\\
    \chi &=& \lambda_V(\chi_0) \in \H^{a-1}(K(X), \mu_n^{\tensor (r-1)})
  \end{eqnarray*}
  are the tame lifts given by Lemma~\ref{tamelift}.

  In order to show that $\res\circ s = \id$ it is enough to prove that
  $\alpha|_{K\widehat(X)} = \alpha_0$ and $\chi|_{K\widehat(X)} = \chi_0$.
  But this follows from the functoriality of cohomology: for instance,
  for $\alpha_0$, let $U_0$ be an open set on which $\alpha_0$ is
  defined (i.e., $\alpha_0$ belongs to the image of $\H^a(U_0,
  \mu_n^{\tensor r}) \to \H^a(k(X_0), \mu_n^{\tensor r})$), let $D_0 =
  X_0 - U_0$, let $D$ be the $V$-lift of $D_0$, and let $U= X-D$.  Observe that
  the generic point of $X_0$ belongs to $U$ so that the natural map
  $\H^a(U, \mu_n^{\tensor r}) \to \H^a(K(X), \mu_n^{\tensor r})$ factors
  through $\H^a(A, \mu_n^{\tensor r})$.  Consequently we have a
  commutative diagram
  \begin{diagram}
    \H^a(U_0, \mu_n^{\tensor r}) & \lTo^{\res}_{\approx} & \H^a(U, \mu_n^{\tensor
      r})\\
    \dTo && \dTo &\rdTo\\
    \H^a(k(X_0), \mu_n^{\tensor r}) & \lTo& \H^a(A, \mu_n^{\tensor
      r}) & \rTo &  \H^a(K(X), \mu_n^{\tensor r})\\
     &\luTo^\approx& \dTo&&\dTo\\
     && \H^a(\hat A, \mu_n^{\tensor r}) & \rTo &  \H^a(K\widehat(X), \mu_n^{\tensor r})
  \end{diagram}
  and $\alpha_0$, viewed as an element of $\H^a(K\widehat(X),
  \mu_n^{\tensor r})$, is obtained by following the path given by
  $U_0$, $k(X_0)$, $\hat A$, and $K\widehat(X)$, while $\alpha|_{K\widehat(X)}$
  can be obtained by following the path given by $U_0$, $U$, $K(X)$,
  and $K\widehat(X)$.  Both paths yield the same element, so this completes
  the proof.
\end{proof}

\subsection{The index does not change}
In section \ref{splitting} we constructed $s = s_{V, \pi}\colon
\H^a(K\widehat(X), \mu_n^{\tensor r}) \to \H^a(K(X), \mu_n^{\tensor r})$
a map splitting the restriction.  In particular, since
$$\Br(K(X))' = \injlim_{n \not \equiv 0 \pmod p} \H^2(K(X), \mu_n)$$
and similarly for $\Br(K\widehat(X))'$, we automatically obtain a map
$$s = s_{V, \pi} \colon \Br(K\widehat(X))' \to \Br(K(X))'$$
that also splits the restriction.  In this section we show that this
map preserves the index.  First let us recall some facts about Brauer
groups of regular schemes.

\begin{lemma} Let $X$ be an integral regular scheme of dimension at
  most~2. \label{Brauerlemma}
  \begin{enumerate}
    \def\theenumi{\roman{enumi}}

  \item The Brauer group $\Br(X)$ of classes of Azumaya algebras on
    $X$ coincides with the cohomological Brauer group $\H^2(X, \GG_m)$.

  \item There is an exact sequence
    $$0 \rTo \Br(X)' \rTo \Br(K(X))'
        \rTo^{\bigoplus \partial_D} \bigoplus_D \H^1(K(D), \QQ/\ZZ)'
    $$
    where $D$ runs over all irreducible Weil (or Cartier) divisors of
    $X$.

  \item If $X$ is projective over a henselian ring $(A, \ideal{m}, k)$
    and the special fiber $X_0 \df = X \times_{\Spec A} \Spec k$ has
    dimension at most~1 then
    $$\Br(X) = \Br(X_0)$$
    In particular, if $X_0$ is a projective smooth curve over a finite
    field $k$ then both groups are trivial.
  \end{enumerate}
\end{lemma}

\begin{proof}
  For (i), see \cite{Milne} IV.2.16.  The injectivity of $\Br(X) \to
  \Br(K(X))$ in (ii) is proven in \cite{Milne} IV.2.6, while the exactness in
  the middle term follows from the purity of the Brauer group (see
  \cite{AG} 7.4 or \cite{Milne} IV.2.18~(b), and also \cite{Sa08}, Lemma 6.6).
Finally (iii) is \cite{GB} 3.1
  (see also \cite{CTOP} 1.3 for a proof using proper base change in
  the prime to $p$ case), together with the fact that for any
  projective smooth curve $C$ over a finite field we have $\Br(C) =
  0$, as follows by comparing the sequence in (ii) with the one from
  Class Field Theory (see \cite{GilleSzamuely} 6.5):
  $$0 \rTo \Br(K(C))\rTo^{\bigoplus \partial_P}
    \bigoplus_{P\in C^{(1)}} \H^1(K(P), \QQ/\ZZ) \rTo^{\sum} \QQ/\ZZ \rTo 0
  $$
  (here $P$ runs over all irreducible Weil divisors of $C$, namely,
  over all its closed points).
\end{proof}

Now we are ready to show

\begin{theorem} The map
$$s = s_{V, \pi} \colon \Br(K\widehat(X))' \to \Br(K(X))'$$
preserves the index.\label{t5}
\end{theorem}

\begin{proof} Let $n$ be prime to $p$.  Given an arbitrary element
$$\hat \gamma = \alpha_0 +  (\chi_0,\pi) \in
  {}_n\hspace{-.03in}\Br(K\widehat(X)) = \H^2(K\widehat(X), \mu_n),
$$
where $\alpha_0 \in {}_n\hspace{-.03in}\Br(k(X_0)) = \H^2(k(X_0), \mu_n)$ and
$\chi_0\in \H^1(k(X_0), \ZZ/n)$, let
$$\gamma = s (\hat \gamma) = \alpha +  (\chi,\pi) \in
  {}_n\hspace{-.03in}\Br(K(X)) = \H^2(K(X), \mu_n)
$$
where $\alpha = \lambda_V (\alpha_0) \in \H^2(K(X),\mu_n)$ and $\chi = \lambda_V(\chi_0) \in \H^1(K(X), \ZZ/n)$ are the
tame lifts of $\alpha_0$ and $\chi_0$.

Since $\res \circ s = \id$, we have that $\res \gamma = \hat \gamma$
and therefore $\ind \hat \gamma  \mid \ind \gamma$.  To prove that
$\ind \gamma \mid \ind \hat \gamma$ we now construct a splitting field
for $\gamma$ of degree $\ind \hat \gamma$ over $K(X)$.

The character $\chi_0$ defines a cyclic extension $L$ of $k(X_0)$ of
degree equal to the order $|\chi_0|$.  Since $k$ is perfect, the normalization $Y_0$ of $X_0$ in $L$ is a smooth curve over $k$, tamely ramified over
$X_0$ (since $|\chi_0|$ is prime to $p =\chara k$) along some mark
$D_0$ of $X_0$ (the ramification locus of $\chi_0$).  By the
Nakayama-Witt index formula (see \cite{JW} 5.15(a)) we have that
$$\ind \hat \gamma = |\chi_0| \cdot \ind (\alpha_0|_{k(Y_0)})$$
But since $k(Y_0)$ is a global field, the Albert-Brauer-Hasse-Noether 
theorem (\cite{neukirch}, Corollary 9.2.3, p. 461, and the functoriality of Corollary 9.1.8, p. 458)
tells us that $\alpha_0|_{k(Y_0)}$ is cyclic, hence there is a cyclic
extension of $k(Y_0)$ of degree $\ind (\alpha_0|_{k(Y_0)})$ that
splits $\alpha_0|_{k(Y_0)}$.  Corresponding to this extension there is
a cyclic cover $Z_0$ of $Y_0$, tamely ramified along some mark $E_0$
of $Y_0$.

Let $D\subset X$ be the $V$-lift of $D_0$.  Let $\rho\colon Y \to X$
be the tamely ramified cover of $(X, D)$ lifting the tamely ramified
cover $\rho_0 \colon Y_0 \to X_0$ of $(X_0, D_0)$, as in
Theorem~\ref{equivcat}.  Now by Lemma~\ref{specialval}~(ii) the set
$(\rho^{-1} V)_{\rm red}$ defines a choice of marks on $Y$ lifting the
closed points of $Y_0$.  Let $E$ be the mark on $Y$ that lifts $E_0$ and whose
support belongs to $(\rho^{-1} V)_{\rm red}$.  Finally define
$\sigma\colon Z \to Y$ to be the tamely ramified cover of $(Y, E)$
lifting the tamely ramified cover $\sigma_0\colon Z_0 \to Y_0$ of
$(Y_0, E_0)$.  Since
\begin{eqnarray*}
[K(Z): K(X)] &=& [K(Z): K(Y)] \cdot [K(Y): K(X)]\\
  &=& [k(Z_0): k(Y_0)] \cdot [k(Y_0): k(X_0)]\\
  &=& \ind (\alpha_0|_{k(Y_0)}) \cdot |\chi_0|\\
  &=& \ind \hat \gamma
\end{eqnarray*}
it is enough to show that $K(Z)$ splits $\gamma$.

Since $Z$ is integral and regular of dimension~2, to show that
$\gamma|_{K(Z)} = 0$ it is enough to show, by Lemma \ref{Brauerlemma},
that $\gamma|_{K(Z)}$ is unramified with respect to all Weil divisors
on $Z$.  On the other hand, $K(Y)$ splits $\chi$ by
Lemma~\ref{uglylemma}, hence $\gamma|_{K(Z)} = \alpha|_{K(Z)}$ and it
remains to show $\alpha|_{K(Z)}$ is unramified with respect to the
Weil divisors on $Z$.  Moreover, by the construction of $\lambda_V$ in
the proof of Lemma \ref{tamelift}, $\alpha \in \H^2(U, \mu_n)$ for
some open set $U \subset X$ that is the complement of a mark with
support in $V$.  Consequently, $\alpha$ only ramifies along marks in
$V$.

Let $D'=D\cup \rho(E)$ where $\rho(E)$ is the image of the mark
  $E$.  By our choice of $E$, $\rho(E) \subset V$ and hence $D'
  \subset V$.  We now have that the composition $\rho\circ \sigma
  \colon Z \to X$ is a tamely ramified cover of $(X,D')$, which is
  finite and flat (Lemma~\ref{specialval}~(i)).  Therefore the image
  of any irreducible Weil divisor $F$ in $Z$ is also a Weil divisor
  $G$ in $X$ by \cite{Liu} IV.3.14 (that is, it cannot ``contract'' to
  a closed point).  Moreover if $G \subset V$ then since $D' \subset
  V$ either $G\subset D'$ or $G\cap D'=\emptyset$, and by
  Lemma~\ref{specialval}~(ii) $F$ is also a mark.  Therefore, since
  the ramification locus of $\alpha$ is contained in $V$, it is enough
  to show that $\alpha|_{K(Z)}$ is unramified at all marks lying over
  marks in $V$. 

Let $F$ be an irreducible mark on $Z$ lying over an irreducible
  mark $G$ on $X$ whose support belongs to $V$.  Since $G\subset D'$
  or $G\cap D'=\emptyset$, by Lemma \ref{specialval}~(iii) the
  ramification degree $e$ of $v_{F}$ over $v_G$ equals the
  ramification degree of $v_{F_0}$ over $v_{G_0}$.  By
  Lemma~\ref{tamelift} and the functorial behavior of residue
  maps under finite extensions, we have a diagram, commutative up to
  sign,

\begin{diagram}
{}_n\hspace{-.03in}\Br(k(X_0))& \rTo^{\lambda_V}& {}_n\hspace{-.03in}\Br(K(X))& \rTo^{\res}& {}_n\hspace{-.03in}\Br(K(Z))\\
\dTo^{\partial_{G_0}} &\text{\fbox{1}}& \dTo^{\partial_G} & \text{\fbox{2}}& \dTo_{\partial_F}\\
\H^1(k(G_0), \ZZ/n)&\rInto &\H^1(K(G), \ZZ/n) & \rTo^{e \cdot \res}& \H^1(K(F), \ZZ/n)\\
   && \uInto  &\text{\fbox{3}}& \uInto \\
&& \H^1(k(G_0), \ZZ/n) & \rTo^{e \cdot \res} & \H^1(k(F_0), \ZZ/n)\\
   && \uTo^{\partial_{G_0}}  &\text{\fbox{4}}& \uTo_{\partial_{F_0}} \\
&& {}_n\hspace{-.03in}\Br(k(X_0)) &\rTo^{\res} & {}_n\hspace{-.03in}\Br(k(Z_0))
\end{diagram}
Here we view $\H^1(k(G_0), \ZZ/n) = \H^1(G, \ZZ/n)$ as the subgroup of
unramified characters of $\H^1(K(G), \ZZ/n)$, and similarly
$\H^1(k(F_0), \ZZ/n) = \H^1(F, \ZZ/n) \subset \H^1(K(F), \ZZ/n)$.

If $\alpha_0 \in {}_n\hspace{-.03in}\Br(k(X_0))$, we
obtain
$$\partial_F (\alpha|_{K(Z)}) = \pm e \cdot \partial_{G_0} (\alpha_0) |_{K(F)}
  \in \H^1(K(F), \ZZ/n)
$$
from squares \fbox{1} + \fbox{2}, and we obtain 
$$\partial_{F_0} (\alpha_0|_{k(Z_0)}) = \pm e \cdot \partial_{G_0} (\alpha_0) |_{k(F_0)}
  \in \H^1(k(F_0), \ZZ/n)
$$
from square \fbox{4}.  Hence $\partial_F (\alpha|_{K(Z)}) = \pm \partial_{F_0} (\alpha_0|_{k(Z_0)})$ by square \fbox{3},
which vanishes since $\alpha_0|_{k(Z_0)} = 0$, and we are
done. 
\end{proof}

\section{Indecomposable and noncrossed product division algebras.}
Adopt all notation from Sections 1-3.  In this section we construct indecomposable division algebras over $K(X)$ and noncrossed product algebras over $K(X)$ of prime power index for all primes $q$ with $q\ne p$.  Note that noncrossed product division algebras with index equal to period over $K(X)$ for $X=\P^1_K$ are already known to exist by \cite{Brussel4}.

\subsection{Indecomposable Division Algebras over $K(X)$.}\label{indecomposables}

We construct indecomposable division algebras over $K(X)$ by constructing them over $K\widehat(X)$ and using the splitting $s:\Br(K\widehat(X))'\to \Br(K(X))'$ from Theorem \ref{t5} to lift the Brauer classes to Brauer classes over $K(X)$ whose underlying division algebras are indecomposable.  The construction over $K\widehat(X)$ follows the method in \cite{Brussel6}, where indecomposable division algebras of unequal prime-power index and period are shown to exist over power series fields over number fields.

We start by stating a well known lemma on the invariants of a Brauer class of a global field after a finite extension.  This lemma is helpful in computing the index reduction of the Brauer class after the finite extension.  

\begin{lemma}[see \cite{lf}, XIII, \S3]  Let $\beta \in \Br(F)$ be a Brauer class over a global field $F$.  Let $L/F$ be a finite Galois extension.  Then for all discrete valuations $w$ in $L$ lying over a fixed prime $v$ of $F$, $\inv_{w}(\beta_L) =e_{v} f_{v}\inv_{v}(\beta)$.\label{l1}
\end{lemma}

We now construct indecomposable division algebras over $K\widehat(X)$.

\begin{proposition}
Let $e$ and $i$ be integers satisfying $1\leq e\leq i \leq 2e-1$.  For any prime $q\ne \chara\,k$ there exists a Brauer class $\hat \gamma \in \Br(K\widehat(X))$ satisfying $(\ind(\hat \gamma),\per(\hat \gamma))=(q^i,q^e)$ and whose underlying division algebra is indecomposable.  
\label{Proposition4.1}
\end{proposition}

\proof Let $1\leq t\leq e$ so that $i=2e-t$.  To prove the proposition we produce a Brauer class $\hat \gamma \in \Br(K\widehat(X))$ such that $(\ind(\hat \gamma),\per(\hat \gamma))=(q^{2e-t},q^e)$ and $\ind(q\hat \gamma)=q^{2e-t-1}$.  Since $\ind(\hat \gamma)=q^{2e-t}$ and $\ind(q\hat\gamma)=q^{2e-t-1}$, by \cite[Lemma 3.2]{Saltman-Indecomposable} the division algebra underlying $\hat \gamma$ is indecomposable.  Choose two closed points $x_1$, $x_2 \in X_0$.  Let $v_1$ and $v_2$ be the discrete valuations on $k(X_0)$ corresponding to $x_1$ and $x_2$.  Let $\alpha_0 \in \Br(k(X_0))$ be the Brauer class whose invariants are
\[\begin{array}{ll}
\inv_{v_1}(\alpha_0)=1/q^e\\
\inv_{v_2}(\alpha_0)=-1/q^e
\end{array}\]
and at all other discrete valuations $v$ on $k(X_0)$, $\partial_v(\alpha_0)=0$.  The Brauer class $\alpha_0$ exists by Hasse's residue theorem (\cite[6.5.4]{GilleSzamuely}) and the fact that $k(X_0)$ is a global field.   Let $\xi_{v_i}=\partial_{v_i}(\alpha_0) \in \H^1(k(v_i),\Q/\Z)$.  Let $k(X_0)_{v_i}$ be the completion of $k(X_0)$ at the valuation $v_i$ and choose unramified characters $\theta_{v_i} \in \H^1(k(X_0)_{v_i},\Q/\Z)$ of order $q^t$.  By the Grunwald-Wang theorem there exists a global character $\theta_0\in \H^1(k(X_0),\Q/\Z)$ of order $q^e$ with restrictions $\theta_{v_i}$ at $v_i$ for $i=1,2$.  

Set $\hat\gamma=\alpha_0+(\theta_0,\pi) \in \Br(K\widehat(X))$, an element with period $q^e$.  We claim that $\ind(\hat\gamma)=q^{2e-t}$ and $\ind(q\hat\gamma)=q^{2e-t-1}$.  By the Nakayama-Witt index formula (see \cite{JW}
5.15(a)) we have $\ind \hat \gamma = |\theta_0| \cdot \ind (\alpha_0|_{k(X_0)(\theta_0)})$ where $\alpha_0|_{k(X_0)(\theta_0)}$ is the restriction of $\alpha_0$ to $k(X_0)(\theta_0)$, the finite extension defined by the character $\theta_0$.  By construction, $|\theta_0|=q^e$ so it is only left to show that $\ind(\alpha_0|_{\theta_0})=q^{e-t}$.  Since $k(X_0)(\theta_0)$ is a finite extension of $k(X_0)$, $k(X_0)(\theta_0)$ is a global field and

\[\ind(\alpha_0|_{k(X_0)(\theta_0)})=\per(\alpha_0|_{k(X_0)(\theta_0)})=\lcm_{w}(|\mathrm{inv}_{w}(\alpha_0|_{k(X_0)(\theta_0)})|)\]
where the least common multiple is taken over all discrete valuations $w$ of $k(X_0)(\theta_0)$.  This shows, by our assumptions on $\alpha_0$, that for all discrete valuations $w$ of $k(X_0)(\theta_0)$,

\[\mathrm{inv}_{w}(\alpha_0|_{k(X_0)(\theta_0)})=\left\{\begin{array}{cl}0, &\textrm{ if }w\textrm{ does not lie over }v_i\textrm{ for }i=1,2\\ 
{\pm|(\theta_0)_{v_i}|}\cdot{q^{-e}}
, & \textrm{ if }w\textrm{ lies over } v_i\textrm{ for }i=1,2 \end{array}\right.\]

By our assumption on $\theta_0$, $|(\theta_0)_{v_i}|=q^t$ for $i=1,2$ and therefore we have $\ind(\alpha_0|_{k(X_0)(\theta_0)})=q^{e-t}$ and $\ind(\hat\gamma)=q^{2e-t}$.

A similar calculation for $q\hat\gamma$ gives $|q\theta_0|=q^{e-1}$ and $\ind(q\alpha_0|_{k(X_0)(q\theta_0)})=q^{e-t}$ since by the same reasoning,

\[\mathrm{inv}_{w}(q\alpha_0|_{k(X_0)(q\theta_0)})=\left\{\begin{array}{cl}0, &\textrm{ if }w\textrm{ does not lie over }v_i\textrm{ for }i=1,2\\ 
{\pm|(q\theta_0)_{v_i}|}\cdot{q^{1-e}}
, & \textrm{ if }w\textrm{ lies over } v_i \textrm{ for }i=1,2\end{array}\right.\]
and $|(q\theta_0)_{v_i}|=q^{t-1}$ for $i=1,2$.  We conclude $\ind(q\hat\gamma)=q^{2e-t-1}$.  \endproof

\begin{theorem}
Let $k$, $X_0$, $K$ and $X$ be as in Section \ref{mainsetup} and let $q$ be a prime with $q\ne \chara\,k$.  Fix integers $e$ and $i$ satisfying $1\leq e\leq i \leq 2e-1$.  Then there exists an indecomposable division algebra $D$ over $K(X)$ satisfying $(\ind(D),\per(D))=(q^i,q^e)$.
\label{t6}
\end{theorem}
\proof
Choose $e$ and $i$ so that $1\leq e\leq i \leq 2e-1$.  By Proposition \ref{Proposition4.1} there exists a Brauer class $\hat\gamma\in \Br(K\widehat(X))$ satisfying $(\ind(\hat\gamma),\per(\hat\gamma))=(q^i,q^e)$ and whose underlying division algebra is indecomposable.  By Theorem \ref{t5}, $\gamma=s(\hat\gamma)\in \Br(K(X))$ has index $q^i$.  Since $s$ is a splitting of the restriction map, we also have $\per(\gamma)=q^e$.  To finish the proof we show the division algebra underlying $\gamma$ is indecomposable.  If $\gamma={\beta_1}+ \beta_2$ with $\ind(\beta_1)\ind(\beta_2)=\ind(\gamma)$ represents a nontrivial decomposition of the division algebra underlying $\gamma$, then $\hat\gamma=\res_{K\widehat(X)}(\beta_1)+\res_{K\widehat(X)}(\beta_2)$.  Since the index can only decrease under $\res_{K\widehat(X)}$ we have $\ind(\hat\gamma)=\ind(\res_{K\widehat(X)}(\beta_1))\ind(\res_{K\widehat(X)}(\beta_2))$.  This represents a nontrivial decomposition of the division algebra underlying $\hat\gamma$, a contradiction.\endproof

\begin{remark}  In the case $X=\P^1_R$, it is not hard to construct $\hat \gamma$ which satisfies the conclusions of Proposition \ref{Proposition4.1} and can be seen to have $\ind(\hat\gamma)=\ind(s(\hat\gamma))$ without the use of Theorem \ref{t5}.  Choose $e,\,i,\,t$ so that $1\leq e \leq i \leq 2e-1$ and $i=2e-t$.  Then, as in the proof of Proposition \ref{Proposition4.1}, choose a single closed point $x_0$ in $X_0=\P^1_k$ of degree $q^{e-t}$.  Let $\xi\in \H^1(k,\ZZ/n)$ be a character of order $q^{2e-t}$ where $n$ is an integer prime to $p$ with $q^i\mid n$.  Set $\alpha_0=(\xi,\pi_{x_0})$ where $\pi_{x_0}$ is the irreducible polynomial corresponding to the closed point $x_0$.  Then, 
\[\partial_x(\alpha_0)=\left\{\begin{array}{ll}0,&\textrm{ if }x \ne x_0 \textrm{ and }x\ne\textrm{ the point at infinity}\\ \res_{k|k(x_0)}\xi,&\textrm{ if } x=x_0\end{array}\right.\]
Set $\theta_0=q^{e-t}\xi \in \H^1(k,\ZZ/n)\hookrightarrow \H^1(k(t),\ZZ/n)$.  Set $\hat\gamma=\alpha_0+ (\theta_0,p)$.  Since $\per(\alpha_0)=|\inv_{x_0}\alpha_0|=q^e$ and $\per((\theta_0,p))=q^e$, $\per(\hat\gamma)=q^e$.  Using the same strategy as Proposition \ref{Proposition4.1} shows that $\ind(\hat\gamma)=q^{2e-t}$ and $\ind(q\hat\gamma)=q^{2e-t-1}$.  Therefore, $\hat\gamma$ satisfies the conclusions of Proposition \ref{Proposition4.1}.  We now check $\ind(s(\hat\gamma))=q^{2e-t}$.  Let $\theta=s(\theta_0)$ which is the unique lift of the constant extension $\theta_0$ to $\H^1(K(t),\ZZ/n)$.  The character $\theta$ defines a $p$-unramified extension $L/K(t)$ of degree $q^e$.  Then, $s(\hat\gamma)_L=(s(\xi),s((\pi_{x_0})))_L+(\theta,p)_L=(s(\xi),s((\pi_{x_0})))_L$.  Thus $\ind(s(\hat\gamma)_L)=\ind((s(\xi),s((\pi_{x_0})))_L)\leq |\xi|/|\theta|=q^{e-t}$ since $L$ is contained in the $p$-unramified constant extension defined by $s(\xi)$ which is a lift of $\xi$.  Therefore, $\ind(s(\hat\gamma))\leq [L:K(t)]q^{e-t}=q^{2e-t}=\ind(\hat\gamma)$.  Since $\ind(s(\hat\gamma))\geq \ind(\hat \gamma)$, we get the equality $\ind(s(\hat\gamma))=\ind(\hat\gamma)$.\label{easyremark}
\end{remark}

\begin{remark}
Set $R=\Z_p$ and $K=\Q_p$ and let $X$ be as in \ref{mainsetup}.  By \cite{Saltman-p-adic-c} the index of any Brauer class in $\Br(K(X))$ divides the square of its period.  Let $q$ be a prime with $q \ne p$.  Theorem \ref{t6} shows that over $K(X)$ there exist indecomposable division algebras of  index-period combination $(q^i,q^e)$ for all $1\leq e\leq i \leq 2e-1$ and all primes $q\ne p$.  In \cite{Suresh}, Suresh builds on the work of \cite{Saltman-cyclicity} to show that if $L/\Q_p(t)$ is a finite extension containing the $q$-th roots of unity, then every element in $\H^2(L,\mu_q)$ is a sum of at most two symbols.  In particular, a division algebra over $L$ of index $q^2$ and period $q$ must be decomposable as it is the sum of two symbols each of index $q$.  In a forthcoming paper by Brussel and Tengan, \cite{BT}, the dependence on an $q$-th root of unity is removed, showing that all division algebras of index-period combination $(q^2,q)$ over $L$ are decomposable for any finite extension $L/\Q_p(t)$.
\end{remark}
\subsection{Noncrossed products over $K(X)$}

In this section we construct noncrossed product division algebras over $K(X)$.  Throughout this section we adopt all notation from Section \ref{mainsetup}.  In particular, $K$ is the fraction field of $R$, a complete discrete valuation ring with uniformizer $\pi$ and residue field $k$, a field of characteristic $p$ and $X$ is a smooth curve over $R$.  We use the same strategy as in Section \ref{indecomposables}, that is, we construct noncrossed products of $q$-power index ($q$ a prime, $q \ne p=\chara k$) over $K\widehat(X)$ and use the splitting $s:\Br(K\widehat(X))' \to \Br(K(X))'$ from Theorem \ref{t5} to lift the noncrossed products to $K(X)$.

The method of constructing the noncrossed products over $K\widehat(X)$ follows the method in \cite{Brussel} where noncrossed products over $\Q(t)$ and $\Q((t))$ are constructed.  In order to mimic the construction in \cite{Brussel} we need only note that both the \v {C}ebotarev density theorem, and the Grunwald-Wang theorem hold for global fields which are characteristic $p$ function fields.  After noting these two facts, the reader can check that the arguments in \cite{Brussel} apply directly to obtain noncrossed products over $K\widehat(X)$ of index and period given below.  

\begin{iesetup}\label{last setup}Let $K$, $R$, $k$, $X$ and $X_0$ be as in Section \ref{mainsetup}.  For any positive integer $a$, let $\epsilon_a$ denote a primitive $a$-th root of unity.  Set $r$ and $s$ to be the maximum integers such that $\mu_{q^r} \subset k(X_0)^\times$ and $\mu_{q^s}\subset k(X_0)(\epsilon_{q^{r+1}})^\times$.  Let $n$ and $m$ be integers such that $n\geq 1$, $n\geq m$, and $n,m \in \{r\}\cup [s,\infty)$.  Let $a$ and $l$ be integers such that $l \geq n+m+1$ and $0\leq a \leq l-n$.  See \cite[p.384-385]{Brussel} for more information regarding these constraints.  \label{ncpsetup}
\end{iesetup}

\begin{theorem}  Let $K$, $R$, $k$, $X$ and $X_0$ be as in Section \ref{mainsetup}.  Let $q$ be a prime, $q \ne p=\chara k$ and let $a$ and $l$ be integers satisfying the properties of \ref{ncpsetup}.  Then there exists noncrossed product division algebras over $K\widehat(X)$ of index $q^{l+a}$ and period $q^l$.
\label{c3}
\end{theorem}

\begin{corollary}  Let $K$, $R$, $k$, $X$, $X_0$, $q$, $a$ and $l$ be as Theorem \ref{c3}.  Then, there exists noncrossed product division algebras over $K(X)$ of index $q^{l+a}$ and period $q^l$.
\label{c4}
\end{corollary}
\proof
Let $\widehat D$ be a noncrossed product over $K\widehat(X)$ of index $q^{l+a}$, period $q^l$.  Let $D$ be the division algebra in the class of $s([\widehat D]) \in \Br(K(X))$.  By Theorem \ref{t5} we know that $\ind(D)=\ind(\widehat D)$.  Assume by way of contradiction that $D$ is a crossed product with maximal Galois subfield $M/K(X)$.  Then $MK\widehat(X)$ splits $\widehat D$, is of degree $\ind(\widehat D)$ and is Galois.  This contradicts the fact that $\widehat D$ is a noncrossed product.\endproof

\begin{remark} Noncrossed products were already known to exist over
  $\Q_p(t)$ by \cite{Brussel4}.  In the noncrossed products of \cite{Brussel4} the index
  always equals the period.  This is not the case in the above
  construction.
\end{remark}

\bibliographystyle{elsarticle-harv} 
\bibliography{NonCrossedProductBib}

\end{document}